\documentclass[a4paper,11pt]{article}
\usepackage[ngerman, english]{babel}
\usepackage[T1]{fontenc}
\usepackage[utf8]{inputenc}
\usepackage{amsmath}
\usepackage{amssymb}
\usepackage{amsthm}
\usepackage{graphicx}
\usepackage{here}
\usepackage{color}
\usepackage{xcolor}
\usepackage{enumerate}
\usepackage{lmodern}
\usepackage{fancyvrb}
\usepackage[plainpages=false]{hyperref}
\usepackage{caption}
\usepackage{subfigure}
\usepackage{epstopdf}
\usepackage{float}

\usepackage{thmtools}
\usepackage{thm-restate}

\theoremstyle{plain}

\newcommand{\ca}{\mathcal}

\newtheorem{defi}{Definition}[section]
\newtheorem{theo}[defi]{Theorem}

\newtheorem{lem}[defi]{Lemma}
\newtheorem{con}[defi]{Conjecture}

\newtheorem{cor}[defi]{Corollary}
\newtheorem{prob}[defi]{Problem}
\newtheorem{ques}[defi]{Question}

\newcounter{claimcount}
\usepackage{etoolbox}
\AtBeginEnvironment{proof}{\setcounter{claimcount}{0}}

\theoremstyle{remark}

\definecolor{mydarkgreen}{rgb}{0.0, 0.6, 0.0}

\title{On the existence of factors intersecting sets of cycles in regular graphs}

\author{J. Goedgebeur$^{1,2}$\thanks{Supported by Internal Funds of KU Leuven and a grant of the Research Foundation Flanders (FWO) with grant number G0AGX24N.}, D. Mattiolo$^1$\thanks{Supported by a Postdoctoral Fellowship of the Research Foundation Flanders (FWO) with grant number 1268323N.
}, G. Mazzuoccolo$^3$, \\ J. Renders$^{1*}$, L. Toffanetti$^3$, I. H. Wolf$^4$\thanks{Supported by a grant of the Deutsche Forschungsgemeinschaft (DFG) with grant number 445863039} \\
	\footnotesize $^1$ Department of Computer Science, KU Leuven Kulak, 8500 Kortrijk, Belgium.\\
 	\footnotesize $^2$ Department of Applied Mathematics, Computer Science and Statistics, \\ \footnotesize Ghent University, 9000 Ghent, Belgium.\\
  \footnotesize
 $^3$ Dipartimento di Scienze Fisiche, Informatiche e Matematiche,\\
 \footnotesize Universit\`a di Modena e Reggio Emilia, 41125 Modena, Italy.
\\
    \footnotesize $^4$ Department of Mathematics, Paderborn University, 33098 Paderborn,
		Germany.}

\date{}

\begin{document}

\maketitle

\begin{abstract}
A recent result by Kardoš, Máčajová and Zerafa
[J.\ Comb.\ Theory, Ser.\ B. 160 (2023) 1--14]  related to the famous Berge-Fulkerson conjecture implies that given an arbitrary set of odd pairwise edge-disjoint cycles, say $\mathcal O$, in a bridgeless cubic graph, there exists a $1$-factor intersecting all cycles in $\mathcal O$ in at least one edge. This remarkable result opens up natural generalizations in the case of an $r$-regular graph $G$ and a $t$-factor $F$, with $r$ and $t$ being positive integers. 
In this paper, we start the study of this problem by proving necessary and sufficient conditions on $G$, $t$ and $r$ to assure the existence of a suitable $F$ for any possible choice of the set $\mathcal O$. 
First of all, we show that $G$ needs to be $2$-connected. Under this additional assumption, we highlight how the ratio $\frac{t}{r}$ seems to play a crucial role in assuring the existence of a $t$-factor $F$ with the required properties by proving that $\frac{t}{r} \geq \frac{1}{3}$ is a further necessary condition. We suspect that this condition is also sufficient, and we confirm it in the case $\frac{t}{r}=\frac{1}{3}$, generalizing the case $t=1$ and $r=3$ proved by Kardoš, Máčajová, Zerafa, and in the case $\frac{t}{r}=\frac{1}{2}$ with $t$ even. Finally, we provide further results for the case where even cycles are included.
\end{abstract}


\section{Introduction}

We consider finite graphs without loops that may have parallel edges. The main question we study in this paper can be briefly stated as follows: \\

{\it "Under what conditions, given an $r$-regular graph $G$ and an arbitrary set of pairwise edge-disjoint cycles of $G$, can we guarantee the existence of  a $t$-factor of $G$ that intersects each of these cycles in at least one edge?"}\\

\noindent where a {\it cycle} is a connected $2$-regular subgraph of $G$.

First of all, we explain how a particular instance of this problem, solved in~\cite{KARDOS20231}, is related to the following long-standing conjecture which dates back to 1971.

\begin{con}[Fulkerson~\cite{fulkerson1971blocking}]\label{con:BFC}
Every bridgeless cubic graph has six perfect matchings such that each edge belongs to exactly two of them.
\end{con}

This conjecture was first proposed by Berge, but Fulkerson~\cite{fulkerson1971blocking} was the one who put it into print (cf.~\cite{seymour1979multi}). Thus, Conjecture~\ref{con:BFC} is also called the Berge-Fulkerson Conjecture. 

It is proved in \cite{ReductionBFC} that Conjecture~\ref{con:BFC} reduces to non-$3$-edge-colorable cubic graphs which are cyclically $5$-edge-connected. Moreover, with the help of a computer, the conjecture was verified for cubic graphs of order at most 36~\cite{brinkmann2013generation}, which can be seen as an indication that the conjecture might be true in general. Nevertheless, a general solution seems to be far away. 

Hence, in order to make further progress, the focus has shifted to weaker statements. The following three conjectures are all implied by the Berge-Fulkerson Conjecture and decrease in their strength, i.e.\ each conjecture is implied by the previous.

\begin{con}[Fan, Raspaud~\cite{fan1994fulkerson}]\label{con:FRC}
Every bridgeless cubic graph has three perfect matchings with an empty intersection.
\end{con}

\begin{con}[Máčajová, Škoviera~\cite{MACAJOVA2005112}, see also~\cite{Kaiser2010}]\label{con:3graph_odd_cut}
Every bridgeless cubic graph has two perfect matchings such that their intersection does not contain an edge-cut of odd cardinality.
\end{con}

\begin{con}[Mazzuoccolo~\cite{MAZZUOCCOLO2013235}]\label{con:S4C}
Every bridgeless cubic graph has two perfect matchings $M_1, M_2$ such that $G-(M_1 \cup M_2)$ is bipartite.
\end{con}

Very recently, the weakest of these conjectures (i.e.\ Conjecture~\ref{con:S4C}) was proved by Kardoš, Máčajová and Zerafa~\cite{KARDOS20231}. 
An equivalent formulation of their main result is the following theorem.

\begin{theo}[Kardoš, Máčajová, Zerafa~\cite{KARDOS20231}]
\label{theo:KMZ_odd_circuits}
Let $G$ be a 2-connected cubic graph. Let $\mathcal{O}$ be a set of pairwise edge-disjoint odd cycles of $G$ and let $e \in E(G)$. Then, there exists a 1-factor $F$ of $G$ such that $e \in E(F)$ and $E(F) \cap E(O) \neq \emptyset$ for every $O \in \mathcal{O}$.
\end{theo}

It is natural to ask whether this result can be extended to graphs of higher degree, thus returning to the initial question we proposed. 

In the present paper, we consider the following instance of our general question, where, in analogy to Theorem~\ref{theo:KMZ_odd_circuits}, we limit our attention to cycles of odd length (note that we do not necessarily require to prescribe an edge of $F$). Results concerning arbitrary cycles can be found in Section~\ref{sec:arbitrary_cycles}.

\begin{ques}\label{ques:t-factor_destroying_odd_circuits}
 Let $r,t$ be integers with $1 \leq t\leq r-2$. Is it true that for every $2$-connected $r$-regular graph $G$ and every set $\mathcal{O}$ of pairwise edge-disjoint odd cycles of $G$ there is a $t$-factor $F$ of $G$ such that $E(F) \cap E(O) \neq \emptyset$ for every $O \in \mathcal{O}$?
\end{ques}

Note that there are combinations of $r$ and $t$ such that there exist $r$-regular graphs that do not have a $t$-factor. In this case, Question~\ref{ques:t-factor_destroying_odd_circuits} turns out to be trivial. Hence, from now on we focus on $r$-regular graphs admitting a $t$-factor (see~\cite{BSW_reg_factors_of_reg_graph}
for necessary and sufficient conditions for the existence of a $t$-factor in an $r$-regular graph).
Furthermore, in Section~\ref{sec:2connectivity} we prove that the assumption that $G$ has no cut-vertices is necessary to assure a positive answer. 

Our main results suggest that the key factor in the study of Question~\ref{ques:t-factor_destroying_odd_circuits} is the ratio $\frac{t}{r}$.
In particular, we show in Section~\ref{sec:tr>13} that a positive answer is possible only if $\frac{t}{r}\geq \frac{1}{3}$ holds. Moreover, in Section~\ref{sec:13} and Section~\ref{sec:12} we show that such a condition is also sufficient when $\frac{t}{r}=\frac{1}{3}$, and when $\frac{t}{r}=\frac{1}{2}$, where $t$ is even, respectively.
We leave it as an intriguing open problem to prove or disprove that the necessary condition $\frac{t}{r} \geq \frac{1}{3}$ is also sufficient in general.

The following two theorems are our main results.

\begin{restatable}{theo}{thmgeneralisationKMZoddcircuits}
\label{theo:generalisation_KMZ_odd_circuits}
Let $t\geq 1$ be an integer and let $G$ be a $2$-connected $3t$-regular graph. Let $\mathcal{O}$ be a set of pairwise edge-disjoint odd cycles of $G$ and let $e \in E(G)$. Then, there exists a $t$-factor $F$ of $G$ such that $e \in E(F)$ and $E(F) \cap E(O)$ is a non-empty matching of $G$ for every $O \in \mathcal{O}$.
\end{restatable}

\begin{restatable}{theo}{thmfourkreg}
\label{thm:4k-reg_case_Odd_circuits}
Let $t \geq 2$ be an even integer and let $G$ be a $2$-connected $2t$-regular graph. Let $\mathcal{O}$ be a set of pairwise edge-disjoint odd cycles of $G$. Then, there exists a $t$-factor $F$ of $G$ such that $E(O) \cap E(F) \neq \emptyset$ and $E(O) \cap (E(G) \setminus E(F)) \neq \emptyset$ for every $O \in \mathcal{O}$.
\end{restatable}

Note that Theorem~\ref{theo:KMZ_odd_circuits} is a special case of Theorem~\ref{theo:generalisation_KMZ_odd_circuits} for $t=1$.

Finally, in Section~\ref{sec:arbitrary_cycles}, we present some further results for a set $\mathcal{O}$ containing arbitrary cycles (not necessarily odd) and in Section~\ref{sec:concluding_remarks} we conclude the paper with some open problems.

\section{First necessary condition: \texorpdfstring{$2$}{2}-connectivity}\label{sec:2connectivity}

In this section we show that the $2$-connectivity assumption in Theorem~\ref{theo:generalisation_KMZ_odd_circuits} and Theorem~\ref{thm:4k-reg_case_Odd_circuits} is necessary.

We denote by $m_G(u,v)$, or simply $m(u,v)$, the number of edges connecting the vertices $u$ and $v$ in a graph $G$. For a subset $X \subseteq V(G)$, the set of edges with exactly one end-vertex in $X$ is denoted by $\partial_G(X)$. For convenience, if $X$ consists of a single vertex $x$ we use the notation $\partial_G(x)$ by omitting the set-brackets.

\begin{theo}\label{theo:1_connected_examples}
Let $r,t$ be integers with $1 \leq t\leq r-2$. Then, there exists an $r$-regular graph $G$ having a $t$-factor and a set $\mathcal{O}$ of pairwise edge-disjoint odd cycles of $G$ such that every $t$-factor of $G$ is edge-disjoint with at least one cycle of $\mathcal{O}$.
\end{theo}

\begin{proof}

Let $r,t$ be integers with $1 \leq t\leq r-2$. We argue according to the parity of $r$ and $t$. First, we construct a graph $G_{r}$ and a set ${\mathcal O}_r$ in the case when $r$ and $t$ are even. After that, we obtain all other cases (with the exception of the case $t=1$, which will be treated separately) by adding to $G_r$ some suitable gadgets and choosing the same set ${\mathcal O}_r$ in the subgraph isomorphic to $G_r$ of the resulting graph.

		\underline{Case 1: $r\geq 4$ even, $t \geq 2$ even.}\\
		Let $J$ be the graph shown in Figure~\ref{fig:J}.
		\begin{figure}[!htbp]
		\centering
		\includegraphics[width=7cm]{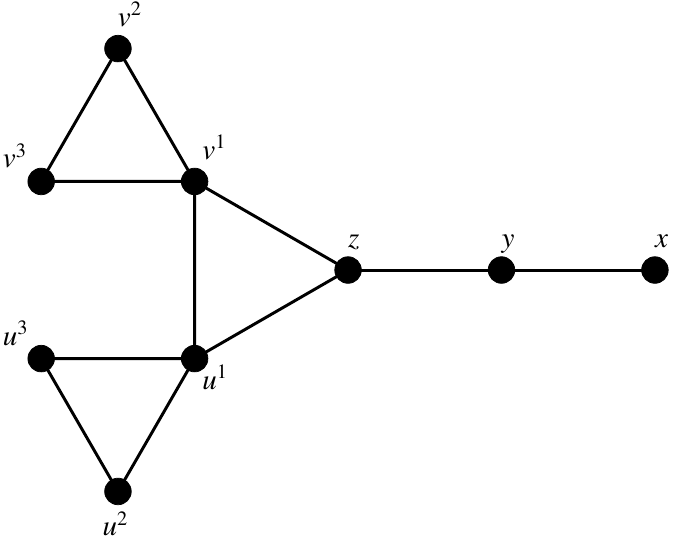}
		\caption{The graph $J$ introduced in the proof of Theorem~\ref{theo:1_connected_examples} (in the case of even $r$ and $t$).}\label{fig:J}
		\end{figure}
		We will refer to the labels in the figure for its vertices and edges. 
		Let $J'$ be the multigraph whose underlying simple graph is $J$ and such that $m(x,y)=2,\ m(y,z)=r-2,\ m(z,v^1)=m(z,u^1)=m(v^1,u^1)=1,\ m(v^1,v^2)=m(v^1,v^3)=m(u^1,u^2)=m(u^1,u^3)=\frac{r-2}{2}$ and $m(u^2,u^3)=m(v^2,v^3)=\frac{r+2}{2}$.	
		Let $G_r$ be the graph obtained as follows. Take $\frac{r}{2}$ distinct copies $J'_1, J'_2,...,J'_{\frac{r}{2}}$ of the graph $J'$; in each copy we label the vertices accordingly by using a lower index. Identify $x_1, \ldots , x_{\frac{r}{2}}$; with a slight abuse of notation we denote the new vertex again by $x$ (see Figure~\ref{fig:rt_even}).
        
  	\begin{figure}[!htbp]
		\centering
		\includegraphics[width=7cm]{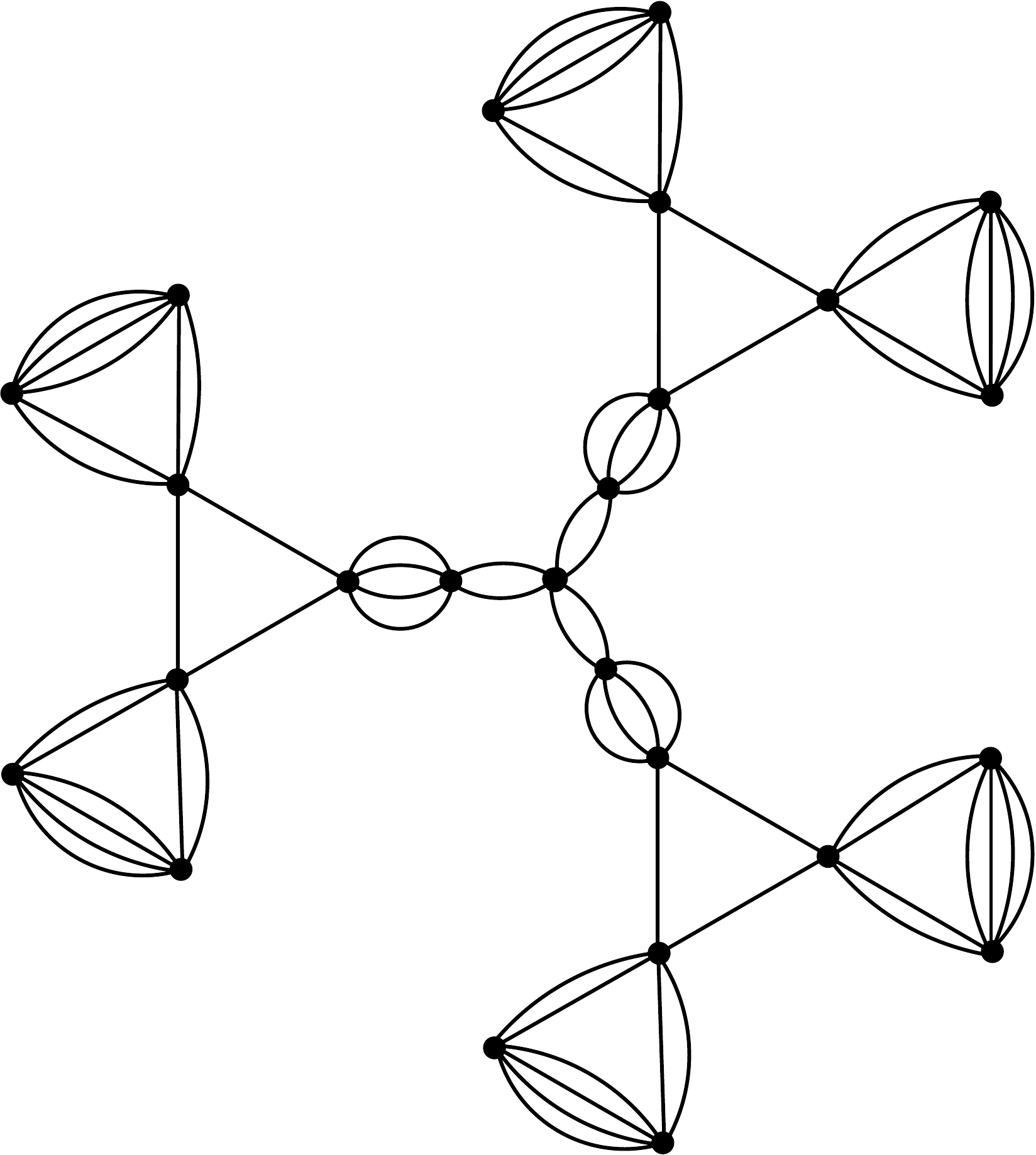}
		\caption{The graph $G_6$ introduced in the proof of the Theorem~\ref{theo:1_connected_examples} (in the case $r=6$ and $t$ even).}\label{fig:rt_even}
		\end{figure}
The graph $G_{r}$ is $r$-regular and, by the Petersen 2-factor Theorem, it admits a $t$-factor since $r,t$ are both even.
		
		Let ${\mathcal{O}}_r=\{O_1,...,O_{\frac{r}{2}}\}$ be the set of vertex-disjoint cycles where, for each $i \in \{1,...,\frac{r}{2}\}$, $O_i$ denotes the $3$-cycle of $J'_i$ induced by the edges $z_iv_i^1, z_iu_i^1$ and $v_i^1u_i^1$. Suppose, by contradiction, that there exists a $t$-factor $F$ of $G_r$ intersecting every cycle in ${\mathcal{O}}_r$. Let $i \in \{1,\ldots ,\frac{r}{2}\}$. Since $m_F(v_i^1,v_i^2)=m_F(v_i^1,v_i^3)$ holds and $t$ is even, we conclude that $|E(F) \cap \{z_iv_i^1,v_i^1u_i^1\}|$ is even. For the same reason, $|E(F) \cap \{z_iu_i^1,v_i^1u_i^1\}|$ is also even. Hence, either $\{z_iv_i^1,z_iu_i^1,v_i^1u_i^1\} \subseteq E(F)$ or $\{z_iv_i^1,z_iu_i^1,v_i^1u_i^1\} \cap E(F)=\emptyset$. Since $F$ intersects each cycle in ${\mathcal{O}}_r$, it is necessarily $\{z_iv_i^1,z_iu_i^1,v_i^1u_i^1\} \subseteq E(F)$. Thus, we have $m_F(x,y_i)=2$. Since $i$ was arbitrarily chosen, $F$ contains every edge incident to $x$, a contradiction.
		
		\underline{Case 2: $r\ge6$ is even, $t\ge3$ is odd.}\\
		Note that $t\le r-3$ in this case. Let $P$ be the \emph{paw graph} depicted in Figure~\ref{fig:paw}. In the rest of the proof we will refer to the labels in the figure for its vertices and edges.
		\begin{figure}[!htbp]
			\centering
			\includegraphics[width=5cm]{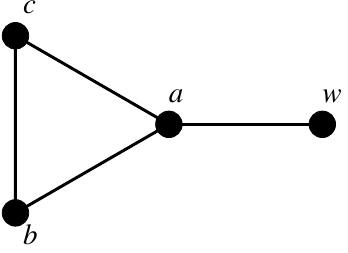}
			\caption{The paw graph introduced in the proof of the Theorem~\ref{theo:1_connected_examples} (in the case of even $r$ and odd $t$).}\label{fig:paw}
		\end{figure}
Let $P_1$ be the multigraph whose underlying simple graph is $P$ and such that $m(w,a)=2,\ m(a,b)=m(a,c)=\frac{r-2}{2}$ and $m(b,c)=\frac{r+2}{2}$.
Let $H_r$ be the $r$-regular graph obtained from $G_{r-2}$ (defined in Case 1) as follows. For each $v \in V(G_r)$ take a copy of $P_1$ and identify $v$ and the vertex corresponding to $w$.

Note that $H_r$ has two disjoint perfect matchings $M_1, M_2$, since $P_1$ has two disjoint perfect matchings. Moreover, $H_r-(M_1 \cup M_2)$ has a $(t-1)$-factor $F'$, since it is a graph of even regularity and $t$ is odd. Hence, $F'+M_1$ is $t$-factor of $H_r$.
  
		Let ${\mathcal{O}}_{r-2}$  be the set of vertex-disjoint cycles of the subgraph of $H_r$ isomorphic to $G_{r-2}$ that is defined in Case 1. By contradiction, suppose there exists a $t$-factor $F$ of $H_r$ intersecting every cycle of ${\mathcal{O}}_{r-2}$. By construction, $\partial_{H_r}(V(G_{r-2}))$ consists of $|V(G_{r-2})|$ pairs of parallel edges; each of them form a 2-edge-cut. Since $t$ is odd, $F$ contains exactly one edge of each pair by parity reasons. As a consequence, $E(F)$ induces a $(t-1)$-factor in $G_{r-2}$ intersecting all cycles in ${\mathcal{O}}_{r-2}$, a contradiction.

		\underline{Case 3: $r \geq 5$ odd and $t \geq 2$.}\\
		Let $P_2$ be the multigraph whose underlying simple graph is $P$ and such that $m(w,a)=1,\ m(a,b)=m(a,c)=\frac{r-1}{2}$ and $m(b,c)=\frac{r+1}{2}$.
Let $H_r'$ be the $r$-regular graph obtained from $G_{r-1}$ (defined in Case 1) as follows. For each $v \in V(G_r)$ take a copy of $P_2$ and identify $v$ and the vertex corresponding to $w$.
  
       Note that $H_r'$ has a perfect matching $M$ since $P_2$ has one. Then $H_r'-M$ has an $l$-factor for all $l\in\{2,4,\dots,r-3\}$. We conclude that $H_r'$ as a $t$-factor.
		
		Let ${\mathcal{O}}_{r-1}$ be the set of vertex-disjoint $3$-cycles in the subgraph of $H_r'$ isomorphic to $G_{r-1}$ that is defined in Case 1. By contradiction, suppose there exists a $t$-factor $F$ of $H_r'$ intersecting every cycle of ${\mathcal{O}}_{r-1}$. By construction, $\partial_{H_r'}(V(G_{r-1}))$ consists of $|V(G_{r-1})|$ bridges. By parity reasons, if $t$ is odd, then $F$ contains all of them; if $t$ is even, then $F$ contains none of them. In the first case, $E(F)$ induces a $(t-1)$-factor in $G_{r-1}$ intersecting all cycles in ${\mathcal{O}}_{r-1}$, a contradiction. In the second case, $E(F)$ induces a $t$-factor in $G_{r-1}$ intersecting all cycles in ${\mathcal{O}}_{r-1}$, a contradiction again.
		
		\underline{Case 4: $t=1$.}\\
		Let $W$ be the graph shown in Figure~\ref{fig:W}. We will refer to the labels in the figure for its vertices and edges.
		\begin{figure}[!htbp]
			\centering
			\includegraphics[width=8cm]{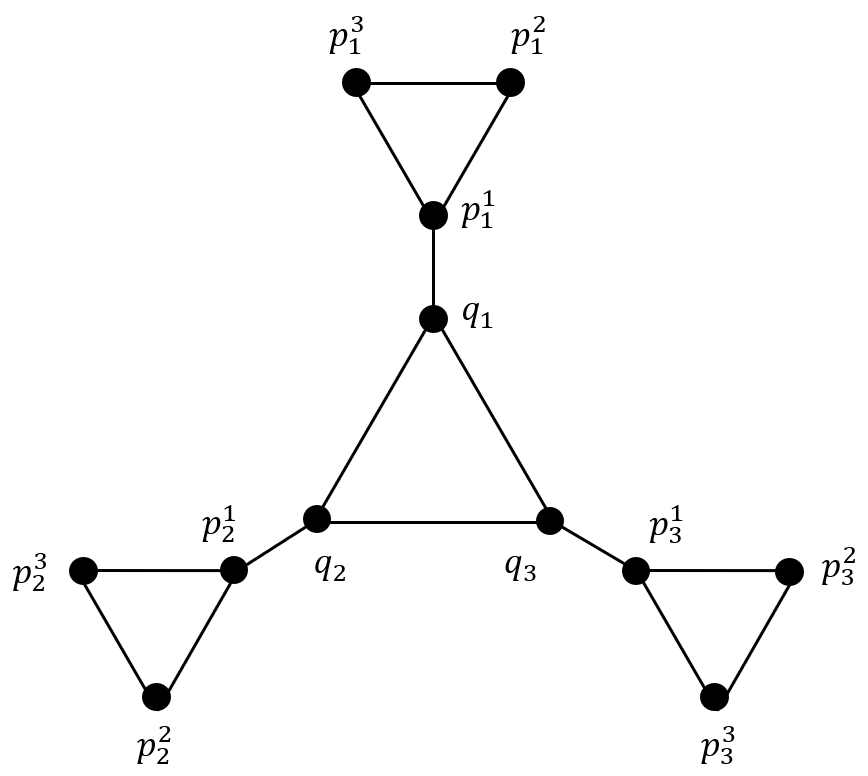}
			\caption{The graph $W$ introduced in the proof of the Theorem~\ref{theo:1_connected_examples} (in the case $t=1$).}\label{fig:W}
		\end{figure}
Let $W_r$ be the $r$-regular multigraph whose underlying simple graph is $W$ and such that $m(q_1,q_2)=m(q_1,q_3)=m(q_2,q_3)=1$ and $m(q_i,p_i^1)=r-2,\ m(p_i^2,p_i^3)=r-1,\ m(p_i^1,p_i^2)=m(p_i^1,p_i^3)=1$ for each $i \in \{1,2,3\}$. Let $\mathcal{O}$ be the set consisting of the $3$-cycle $O$ induced by the edges $q_1q_2,q_1q_3,q_2q_3$. Observe that $W$ has a unique $1$-factor, which does not contain an edge of $O$. As a consequence, $W_r$ has a $1$-factor but no $1$-factor of $W_r$ contains an edge of $O$.
\end{proof}

\section{Second necessary condition: \texorpdfstring{$\frac{t}{r} \geq \frac{1}{3}$}{t/r >= 1/3}}\label{sec:tr>13}

In this section, we obtain another necessary condition to have a positive answer to Question~\ref{ques:t-factor_destroying_odd_circuits}.
Indeed, the following theorem gives a negative answer to Question~\ref{ques:t-factor_destroying_odd_circuits} for all pairs $r,t$ such that $\frac{t}{r}<\frac{1}{3}$. It turns out  that the necessary condition $\frac{t}{r}\geq \frac{1}{3}$ remains valid even if we consider connectivity assumptions stronger than the $2$-connectivity proved in the previous section.

\begin{theo}
\label{theo:odd_circuits_t<r/3}
For every $r\geq 3$ there is an $r$-connected $r$-regular graph $G$ of even order  
and a set $\ca O$ of pairwise disjoint odd cycles of $G$ with the property that if $F$ is a $t$-factor of $G$ with $E(F) \cap E(O) \neq \emptyset$ for every $O \in \mathcal{O}$, then $t \geq \frac{r}{3}$.
\end{theo}

\begin{proof}
For $i \in \{1,2,3\}$ let $G_i$ be a graph isomorphic to $K_{r-2,r}$, where the two partitions are given by $A_i=\{a_1^i, \ldots, a_r^i\}$ and $B_i=\{b_1^i, \ldots, b_{r-2}^i\}$.

Construct a new graph $G$ from $G_1,G_2,G_3$ by adding the edges $a_j^1 a_j^2$, $a_j^1 a_j^3$ and $a_j^2 a_j^3$ for every $j \in \{1, \ldots, r\}$. Let $\ca O$ be the set of pairwise disjoint triangles of $G$ that are induced by the added edges, i.e.\ $\ca O=\{a_j^1 a_j^2 a_j^3 a_j^1 \colon j \in \{1, \ldots, r\}\}$. The graph $G$ and the set of cycles $\ca O$ are depicted in Figure~\ref{fig:t_small} in the case when $r=4$.
\begin{figure}[!htbp]
\centering
\includegraphics[scale=0.5]{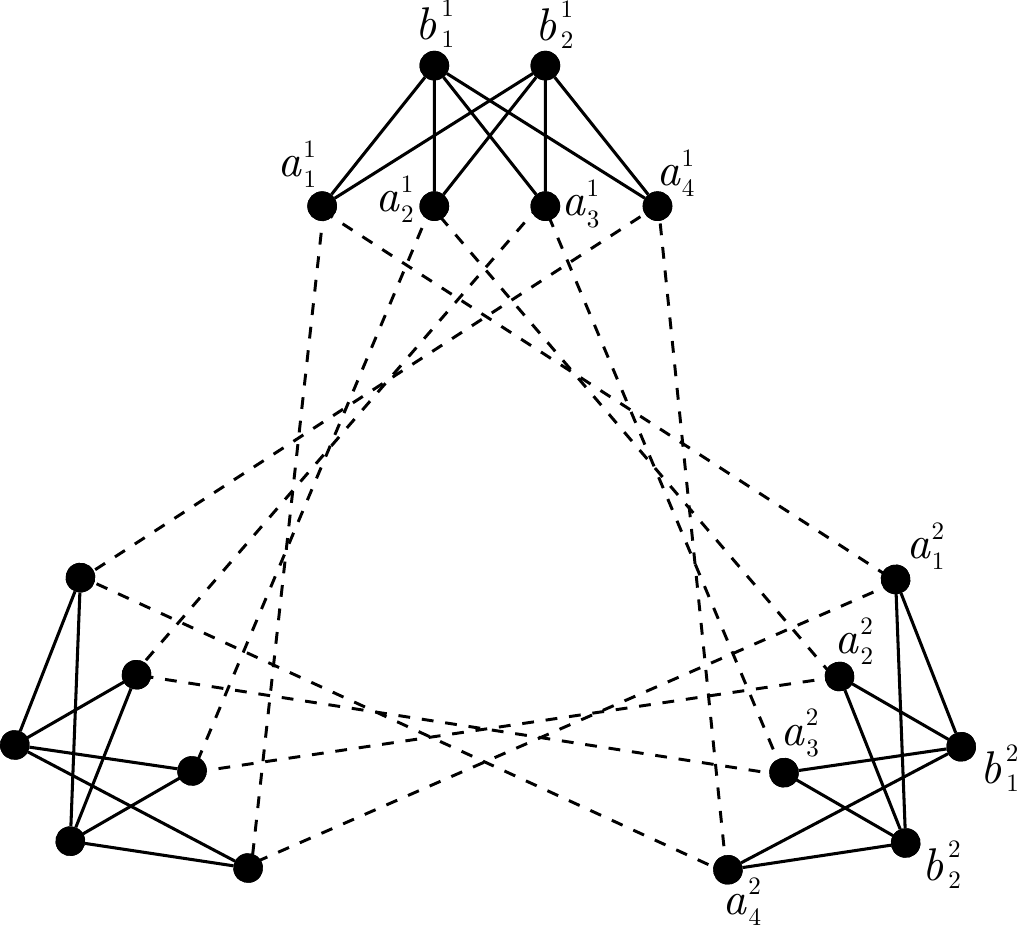}
\caption{The graph $G$ constructed in the proof of Theorem~\ref{theo:odd_circuits_t<r/3} in the case $r=4$. The set $\ca O$ consists of the triangles whose edges are drawn with dashed lines.}
\label{fig:t_small}
\end{figure}
We claim that $G$ and $\ca O$ have the desired properties. By construction, $G$ is an $r$-regular graph of even order. Let $X \subset V(G)$ be a vertex-cut of $G$. We will show that $\lvert X \rvert \geq r$, hence $G$ is $r$-connected. If the subgraph of $G$ induced by $V(G_i) \setminus X$ is connected for every $i \in \{1,2,3\}$, then $X$ contains at least one vertex of every triangle of $\ca O$. As a consequence $|X|\geq r$. Thus, we can assume that w.l.o.g.\ the subgraph of $G$ induced by $V(G_1) \setminus X$ is not connected. Hence, $A_1 \subseteq X$ or $B_1\subseteq X$. In both cases, $|X|\geq r$ since $G-B_1$ does not have a cut-vertex. Therefore, $G$ is $r$-connected.

Finally, let $F$ be a $t$-factor of $G$ with $E(F) \cap E(O) \neq \emptyset$ for every $O \in \mathcal{O}$.
By the construction of $G$, we have $|E(F)\cap \partial_G(B_1 \cup B_2 \cup B_3)|=3(r-2)t$. Furthermore, $E(F)$ contains at least one edge of every triangle in $\ca O$. As a consequence, $3(r-2)t+r \leq |E(F)|=t(3r-3)$, which can be transformed to $\frac{t}{r} \geq \frac{1}3$ by a short calculation.
\end{proof}

We remark that every $r$-connected $r$-regular graph of even order admits a $t$-factor for all $t \in \{1,\ldots, r\}$ (see \cite{BSW_reg_factors_of_reg_graph}), which in particular includes the graphs constructed in the above proof.

The next two sections are devoted to the hardest task of proving some sufficient conditions to have a positive answer to Question~\ref{ques:t-factor_destroying_odd_circuits}.

\section{A first sufficient condition: \texorpdfstring{$\frac{t}{r}=\frac{1}{3}$}{t/r=1/3}}\label{sec:13}

In this section, we prove Theorem~\ref{theo:generalisation_KMZ_odd_circuits}. 
The hardest part of its proof lies in the basic case $r=3$ and $t=1$ proved in \cite{KARDOS20231}. Starting from that result, we generalize to arbitrary values $t,r$ such that $\frac{t}{r}=\frac{1}{3}.$

For a tree $T$, the set of leaves of $T$ is denoted by $Leaf(T)$. An edge of $T$ which is incident with a leaf is called a \emph{pendant} edge.  Moreover, a pendant edge is \emph{lonely} if it is not adjacent to another pendant edge.
We first define a series of sets of trees $\mathcal{T}^1,\mathcal{T}^2, \ldots$ inductively as follows:
\begin{itemize}
\item $\mathcal{T}^1=\{K_{1,3}\}$
\item for every $t>1$, $\mathcal{T}^t$ consists of all trees that can be obtained as follows:
\begin{enumerate}
\item start with a tree $T \in \mathcal{T}^{t-1}$
\item add two copies $H_1, H_2$ of $K_{1,3}$
\item identify $l,l_1$ and $l_2$ to a new vertex, where $l \in Leaf(T)$, $l_1 \in Leaf(H_1)$ and $l_2 \in Leaf(H_2)$.
\end{enumerate}
\end{itemize}

The only graphs in $\ca T^2$ and $\ca T^3$ are depicted in Figure~\ref{fig:three_in_S_2}. 
\begin{figure}[!htbp]
\centering
\includegraphics[scale=0.5]{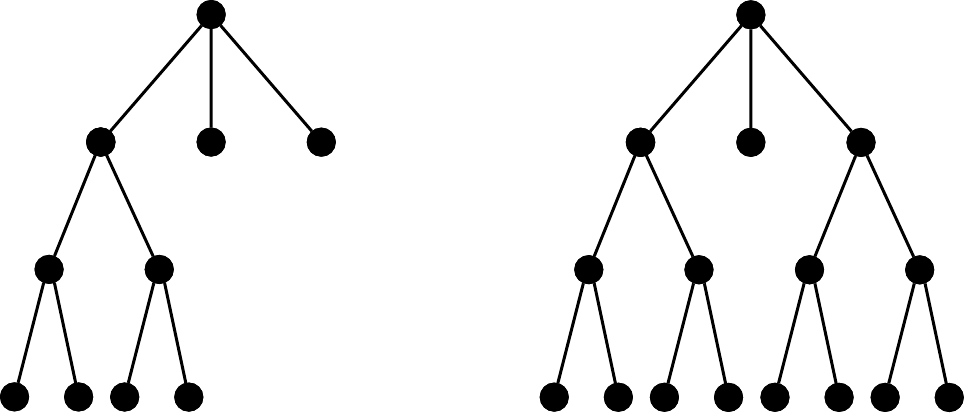}
\caption{The only element of $\ca T^2$ (left) and $\ca T^3$ (right).}
\label{fig:three_in_S_2}
\end{figure}
Note that for every positive integer $t$, every tree of $\mathcal{T}^t$ has exactly $3t$ leaves.
Furthermore, we will use the following two lemmas.

\begin{lem}\label{obs:adjacent_pendant_edges}
For every integer $t\geq 1$, the set $\mathcal{T}^t$ contains a tree with at most one lonely pendant edge.
\end{lem}

\begin{proof}
We prove the statement by induction on $t$. For $t \in \{1,2\}$ the statement is trivially true. Next, let $t \geq 3$ and let $T \in \mathcal{T}^{t-1}$ be a tree with at most one lonely pendant edge. If $T$ has no lonely pendant edge, then every tree in $\mathcal{T}^{t}$ obtained from $T$ has exactly one lonely pendant edge. If $T'$ has one lonely pendant edge $vl$ where $l \in Leaf(T)$, then the tree in $\mathcal{T}^{t}$ obtained from $T$ by applying step 3 (see above definition) on $l$ has no lonely pendant edge.
\end{proof}

\begin{lem}\label{obs:number_of_leafs}
Let $T \in \mathcal{T}^t$ and let $M \subset E(T)$ be a matching such that every vertex of $V(T) \setminus Leaf(T)$ is matched by $M$. Then, exactly $t$ leaves of $T$ are matched by $M$.
\end{lem}

\begin{proof}
We prove the statement by induction on $t$. For $t=1$ the statement is trivially true. Next, assume the statement is true for every $t'<t$. Let $T' \in \mathcal{T}^{t-1}$ and let $T \in \mathcal{T}^{t}$ be obtained from $T'$ by identifying the leaf $l$ of $T'$ and the leaves $l_1,l_2$ one in each of two copies of $K_{1,3}.$  Let $e$ be the unique edge of $T'$ incident to $l$. With a slight abuse of notation, we denote by $e$ also the corresponding edge in $T$. Let $e_1,\ldots,e_4$ be the pendant edges of $T$ corresponding to the edges in the two added copies of $K_{1,3}$. If $e \in M$, then $M$ contains exactly two edges of $\{e_1,\ldots,e_4\}$. Hence, $M$ contains exactly $(t-1)-1+2=t$ pendant edges of $T$ by induction. If $e \notin M$, then $M$ contains exactly one edge of $\{e_1,\ldots,e_4\}$. Thus, $M$ contains exactly $(t-1)+1=t$ pendant edges, again by induction.
\end{proof}

For a vertex $v$ of a graph $G$, and a graph $H$ disjoint from $G$, a new graph $G'$ can be obtained from $G$ as follows: consider the disjoint union of $G$ and $H$; for every edge $e \in E(G)$ incident to $v$, replace the end-vertex $v$ of $e$ by a vertex of $H$; delete $v$.
We say $G'$ is obtained from $G$ by \emph{replacing} $v$ with $H$. Note that there are many different graphs that can be obtained from $G$ by replacing $v$ with $H$; all of them have vertex-set $(V(G)\setminus{v}) \cup V(H)$ and edge-set  $E(G) \cup E(H)$.

We will now prove Theorem~\ref{theo:generalisation_KMZ_odd_circuits}, which we repeat here for the reader's convenience:

\thmgeneralisationKMZoddcircuits*

\begin{proof}
Let $T \in \mathcal{T}^t$ be a tree with at most one lonely pendant edge, which exists by Lemma~\ref{obs:adjacent_pendant_edges}. Recall from the proof of that lemma that such a tree has no lonely pendant edges for $t$ even and exactly one lonely pendant edge for $t>1$ odd. First, we transform $G$ into a new graph $G'$ as follows. For every $v \in V(G)$ replace $v$ by a copy $T_v$ of $T-Leaf(T)$ such  that (1) every vertex of $T_v$ is of degree 3 and (2) if $e,f \in \partial_G(v)$ belong to the same cycle of $\mathcal{O}$, then $e,f$ remain adjacent in the resulting graph $G'$. An example is given in Figure~\ref{fig:Replacement_S_2}.
\begin{figure}[!htbp]
\centering
\includegraphics[scale=0.5]{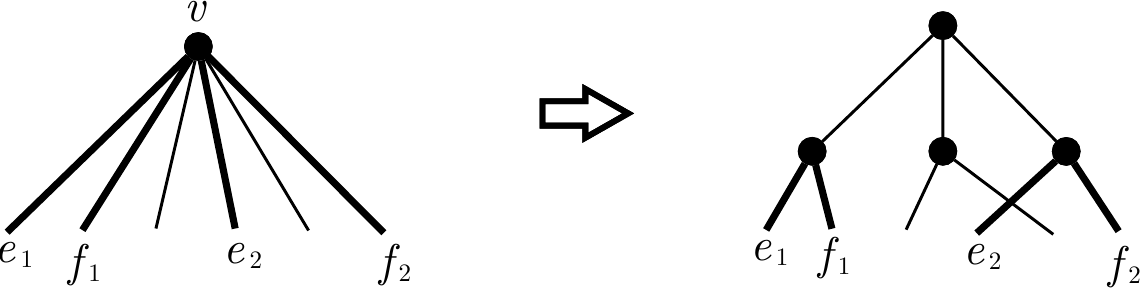}
\caption{The replacement of a vertex $v$ in the proof of Theorem~\ref{theo:generalisation_KMZ_odd_circuits} in the case that $t=2$ and the edges $e_1, f_1$ as well as $e_2,f_2$ belong to the same cycle.}
\label{fig:Replacement_S_2}
\end{figure}

Note that (1) is possible since $T$ has exactly $3t$ leaves; (2) is possible since $T$ has at most one lonely pendant edge.
We obtain a cubic graph $G'$ with vertex-set $\bigcup_{v \in V(G)}V(T_v)$ and edge-set $E(G) \cup \bigcup_{v \in V(G)}E(T_v)$. Furthermore, for every $v \in V(G)$, the graph $G'-V(T_v)$ is connected, since $G$ is $2$-connected. As a consequence, $G'$ is 2-connected. For every $O \in \mathcal{O}$ the subgraph of $G'$ induced by $E(O)$ is an odd cycle in $G'$, which will be denoted by $O'$. Let $\mathcal{O}'=\{O'\colon O \in \mathcal{O}\}$. By Theorem~\ref{theo:KMZ_odd_circuits}, $G'$ has a perfect matching $M$ such that $e \in M$ and $M \cap E(O') \neq \emptyset$ for every $O' \in \mathcal{O}'$. Let $F$ be the subgraph of $G$ induced by the edge-set $M \cap E(G)$. By Lemma~\ref{obs:number_of_leafs}, $|\partial_{G'}(V(T_v)) \cap M|=t$ for every $v \in V(G)$ and hence, $F$ is a $t$-factor of $G$. Furthermore, $E(O') \cap M$ is a non-empty matching in $G'$ for every $O' \in \mathcal{O}'$ and therefore, $E(O) \cap E(F)$ is a non-empty matching in $G$ for every $O \in \mathcal{O}$ by the construction of $G'$. Thus, $F$ has the desired properties.
\end{proof}

Note that if $r$ and $t$ have the same parity, then for every $r$-regular graph $G$ and every $t$-factor $F$ of $G$ the graph $G-E(F)$ can be decomposed into $2$-factors. Thus, for every $t'\in \{t, t+2, \ldots r\}$ the graph $G$ has a $t'$-factor that contains $F$. As a consequence, Theorem~\ref{theo:generalisation_KMZ_odd_circuits} implies the following corollary.

\begin{cor}
\label{cor:3k-reg_case_Odd_circuits}
Let $t\geq 1$ be an integer and let $G$ be a 2-connected $3t$-regular graph. Let $\mathcal{O}$ be a set of pairwise edge-disjoint odd cycles of $G$ and let $e \in E(G)$. Then, for every $l \in \{t, t+2,\ldots,3t\}$ there exists a $l$-factor $F$ of $G$ such that $e \in E(F)$ and $E(F) \cap E(O) \neq \emptyset$ for every $O \in \mathcal{O}$.
\end{cor}

\section{Another sufficient condition: \texorpdfstring{$\frac{t}{r}=\frac{1}{2}$}{t/r=1/2}, \texorpdfstring{$t$}{t} even.}\label{sec:12}

In light of all the previous results, the smallest open instance for Question~\ref{ques:t-factor_destroying_odd_circuits} remains $r=4$ and $t=2$: in this section we give a positive answer for this and other cases by proving Theorem~\ref{thm:4k-reg_case_Odd_circuits}.

Let $G$ be a graph with an \emph{orientation} $D$, i.e.\ a choice of direction for each edge. A cycle $C$ of $G$ is an \emph{oriented} cycle (with respect to $D$), if for every $v \in V(C)$ exactly one edge of $\partial_G(v)\cap E(C)$ is directed towards $v$. Moreover, for a vertex $v \in V(G)$, the number of edges directed towards $v$ is the \emph{indegree} of $v$ with respect to $D$. 

We first prove the following two lemmas.

\begin{lem}
\label{lem:existence_trees}
For every integer $t \geq 2$, there exists a tree $T$ such that $T$ has $2t$ leaves, every other vertex is of degree 3 and $T$ has no lonely pendant edge.
\end{lem}

\begin{proof}
 We argue by induction on the number of leaves. For $t=2$, the tree obtained by attaching two pendant edges to both vertices of $K_2$ has the desired properties. Next, set $t\geq2$ and let $T'$ be a tree that has the properties stated in the lemma. Let $u$ and $v$ be two leaves of $T'$ at mutual distance $2$. Consider four new vertices $u_1,u_2$ and $v_1,v_2$. Construct a new tree $T''$ having vertex-set $V(T') \cup \{u_1,u_2,v_1,v_2\}$ and edge-set $E(T')\cup\{uu_1,uu_2,vv_1,vv_2\}.$ By construction, $T''$ has $2(t+1)$ leaves, every other vertex is of degree 3 and $T''$ has no lonely pendant edge.
\end{proof}

\begin{lem}
\label{lem:orientation_4k-reg_case}
Let $t\geq2$ be an even integer, let $G$ be a $2$-connected $2t$-regular graph and let $\mathcal{O}$ be a set of pairwise edge-disjoint odd cycles of $G$. Then, there exists an orientation of $G$ such that
\begin{itemize}
\item[$(i)$] every vertex has an even indegree,
\item[$(ii)$] no cycle of $\ca O$ is an oriented cycle.
\end{itemize}
\end{lem}

\begin{proof}
Since $G$ is Eulerian, the graph $G-\bigcup_{O \in \mathcal{O}} E(O)$ can be decomposed into cycles. Thus, $G$ has a decomposition $\mathcal{Q}$ into cycles such that $\mathcal{O} \subseteq \mathcal{Q}$. Let $D$ be an orientation of $G$ such that every cycle of $\mathcal{Q}$ is an oriented cycle. We will change the direction of some edges in order to obtain the desired orientation.
 
Let $T$ be a tree such that $T$ has $2t$ leaves, every other vertex is of degree 3 and $T$ has no lonely pendant edge. Such a tree exists by Lemma~\ref{lem:existence_trees} and we note that $T$ as well as $T-Leaf(T)$ are of even order.

Now, we transform $G$ into a cubic graph as follows.
For every $v \in V(G)$ replace $v$ by a copy $T_v$ of $T-Leaf(T)$ such  that (1) every vertex of $T_v$ is of degree 3 and (2) if $e,f \in \partial_G(v)$ belong to the same cycle of $\mathcal{O}$, then $e,f$ remain adjacent in the resulting graph.
We obtain a cubic graph $G'$ with $V(G')=\bigcup_{v \in V(G)} V(T_v)$ and $E(G')=E(G) \cup \bigcup_{v \in V(G)} E(T_v)$. Furthermore, $G'$ is 2-connected, since $G$ is $2$-connected. For every $O \in \mathcal{O}$ the subgraph of $G'$ induced by $E(O)$ is an odd cycle in $G'$, which will be denoted by $O'$. Let $\mathcal{O}'=\{O'\colon O \in \mathcal{O}\}$. Hence, by Theorem~\ref{theo:KMZ_odd_circuits}, $G'$ has a perfect matching $M$ such that $E(O') \cap M \neq \emptyset$ for every $O' \in \mathcal{O}'$.

Now, for every $e \in E(G)$ for which the corresponding edge in $G'$ belongs to $M$, change the direction of $e$ in $D$ to obtain a new orientation $D'$ of $G$. For every $v \in V(G)$, the set $M \cap \partial_{G'}(V(T_v))$ is of even cardinality, since $T_v$ is of even order. Hence, for every $v \in V(G)$ we changed the direction of an even number of edges of $\partial_G(v)$. Thus, $D'$ satisfies $(i)$ since in $D$ every vertex has indegree $t$. Furthermore, for every $O' \in \mathcal{O}'$, the set $E(O') \cap M$ is a non-empty matching in $G'$. As a consequence, $D'$ satisfies $(ii)$.
\end{proof}

Note that Lemma~\ref{lem:orientation_4k-reg_case} does not hold in general for graphs of even regularity. Indeed, if $t$ is odd, then every $2t$-regular graph of odd order has an odd number of edges and hence, does not admit an orientation such that every vertex has an even indegree. Nevertheless, it is unclear to us whether Lemma~\ref{lem:orientation_4k-reg_case} is true for all $2t$-regular graphs with an even number of edges.

We will now prove Theorem~\ref{thm:4k-reg_case_Odd_circuits}, which we repeat here for the reader's convenience.

\thmfourkreg*

\begin{proof}
Consider an orientation $D$ of $G$ that satisfies properties $(i)$ and $(ii)$ of Lemma~\ref{lem:orientation_4k-reg_case}.
For every $v \in V(G)$, split $v$ into $t$ vertices $v_1,\ldots, v_{t}$ of degree $2$ (that is, replace $v$ by a graph $H_v$ consisting of $t$ isolated vertices $v_1, \ldots ,v_{t}$ such that every vertex of $V(H_v)$ is of degree 2 in the resulting graph). We obtain a 2-regular graph $G'$ with $E(G')=E(G)$ and $V(G')=\bigcup_{v \in V(G)} \{v_1, \ldots ,v_{t}\}$. Since every vertex in $G$ has even indegree (with respect to $D)$, this procedure can be done such that (1) for every $v \in V(G)$ and every $i \in \{1,\ldots,t\}$ the two edges incident with $v_i$ in $G'$ are either both directed towards $v$ or both not directed towards $v$ in $G$ and (2) if $O \in \mathcal{O}$, $x \in V(O)$ and $e,f \in E(O)\cap \partial_G(x)$ are such that $e,f$ are either both directed towards $x$ or both not directed towards $x$, then $e,f$ are adjacent in $G'$. By (1), the graph $G'$ is bipartite. Hence, it has a perfect matching $M$. Let $F$ be the subgraph of $G$ induced by the edge set $M$. Observe that $F$ is a $t$-factor of $G$, since $M$ is a perfect matching of $G'$. Furthermore, for every $O \in \mathcal{O}$ there is a $v_O \in V(O)$ such that the two edges in $\partial_G(v_O) \cap E(O)$ are either both directed towards $v_O$ or both not directed towards $v_O$, since no cycle of $\ca O$ is an oriented cycle (with respect to $D$). Thus, by the construction of $G'$, these two edges are adjacent in $G'$, and hence $M$, as well as $E(F)$, contains exactly one of these edges. As a consequence, $E(O) \cap E(F) \neq \emptyset$ and $E(O) \cap (E(G) \setminus E(F)) \neq \emptyset$. Thus, $F$ has the desired properties.
\end{proof}

We remark that in the previous proof the assumption $t$ is even is only used when applying Lemma~\ref{lem:orientation_4k-reg_case}. Hence, if Lemma~\ref{lem:orientation_4k-reg_case} could be extended to all $2t$-regular graphs with an even number of edges, then Theorem~\ref{thm:4k-reg_case_Odd_circuits} is also true for all $2t$-regular graphs with an even number of edges. Furthermore, as is the case of $3t$-regular graphs, we obtain the following corollary.

\begin{cor}
\label{cor:4k-reg_case_Odd_circuits}
Let $t\geq 2$ be an even integer and let $G$ be a 2-connected $2t$-regular graph. Let $\mathcal{O}$ be a set of pairwise edge-disjoint odd cycles of $G$. Then, for every $l \in \{t, t+2,\ldots,2t\}$ there exists a $l$-factor $F$ of $G$ such that $E(F) \cap E(O) \neq \emptyset$ for every $O \in \mathcal{O}$.
\end{cor}

\section{Arbitrary cycles}\label{sec:arbitrary_cycles}

    In addition to addressing other instances of Question~\ref{ques:t-factor_destroying_odd_circuits}, it is also interesting to study the same problem in the case that $\ca O$ consists of pairwise edge-disjoint cycles, regardless of whether they are even or odd. 
    
    First of all, we emphasize that including or excluding $2$-cycles (cycles consisting of two parallel edges) in the set $\mathcal{O}$ radically changes the nature of the problem. For any $r>3$, let $G$ be an $r$-regular $(r-2)$-connected graph of even order without a $1$-factor. Note that the existence of $G$ can be proved by combining the main result in \cite{BSW_reg_factors_of_reg_graph} with the classical Meredith extension operation \cite{Meredith}.  
    Now, consider the graph $2G$ obtained by doubling every edge of $G$. The graph $2G$ is a $2r$-regular $(r-2)$-connected graph and it admits an $(r+1)$-factor (see $\cite{BSW_reg_factors_of_reg_graph}$). 
    Let $\mathcal{O}$ be the set of all $2$-cycles of $2G$ consisting of one edge of the original graph $G$ and its parallel copy. If $2G$ had an $(r+1)$-factor intersecting all cycles in $\mathcal{O}$, then $G$ would have a $1$-factor, a contradiction. Therefore, no $(r+1)$-factor of $2G$ intersects all $2$-cycles of $\mathcal{O}$. In particular, we obtain that a fraction exceeding $\frac{1}{2}$, namely $\frac{r+1}{2r}$,  is not sufficient to guarantee the existence of the required factor: the largest fraction is $\frac{5}{8}$, obtained for $r=4$. 
    
    Furthermore, these examples are the only we have that admit a desired $t$-factor for a given value $t=r$, but not for the larger value $t=r+1$.
    We suspect that such a behavior is due to the presence of $2$-cycles in $\cal O$. Indeed, when $2$-cycles are not permitted, we are not aware of any example of an $r$-regular graph $G$ admitting both a $t$-factor and a $t'$-factor for suitable $t<t'$, where a $t$-factor intersects all members of $\mathcal{O}$, but no $t'$-factor does.
    
    Hence, we state the analogue of Question~\ref{ques:t-factor_destroying_odd_circuits} for arbitrary cycles, with the additional assumption that $\mathcal{O}$ contains only cycles of length at least $3$.

    \begin{ques}
\label{ques:t-factor_destroying_arbitrary_circuits}
 Let $r,t, \kappa$ be positive integers with $1 \leq t\leq r-2$. Is it true that for every $\kappa$-connected $r$-regular graph $G$ and every set $\mathcal{O}$ of pairwise edge-disjoint cycles of length at least $3$ in $G$ there is a $t$-factor $F$ of $G$ such that $E(F) \cap E(O) \neq \emptyset$ for every $O \in \mathcal{O}$?\end{ques}

Theorem~\ref{theo:1_connected_examples} already assures that we need to assume $\kappa \geq 2$ to hope for a general positive answer to the previous question.  

For $\kappa=2$, the following construction shows that a necessary condition to have a positive answer is $\frac{t}{r}\geq \frac12$. Consider the graph $2kK_2$ consisting of $2k$ parallel edges $e_1,...,e_{2k}$ between two vertices $u_1$ and $u_2$. Subdivide each edge $e_i$ once; denote the new vertex by $v_i$. Add $2k-2$ parallel edges between $v_i$ and $v_{i+1}$ for all odd $i\in \{1,3,..., 2k-1\}$. The resulting graph is $2$-connected and $2k$-regular. Consider the set of $k$ 4-cycles with vertices $u_1v_iu_2v_{i+1}$ for all odd $i \in \{1,3,..., 2k-1\}$. If a $t$-factor intersects all of them, then it must contain at least two edges in each of the $k$ $4$-cycles. Hence, it is easy to check that $t \geq k$, i.e.\ $\frac{t}{r} \geq \frac{1}{2}$. 

 Thus, in order to obtain results similar to Theorem~\ref{theo:generalisation_KMZ_odd_circuits} and Theorem~\ref{thm:4k-reg_case_Odd_circuits}, we need to assume that $G$ is $3$-connected. 

Kardoš, Máčajová and Zerafa~\cite{kardos2025threecuts} extended Theorem~\ref{theo:KMZ_odd_circuits} to arbitrary cycles by strengthening connectivity assumptions, thus proving the following statement.

\begin{theo}[Kardoš, Máčajová, Zerafa~\cite{kardos2025threecuts}]
\label{theo:KMZ_1+factor_acyclic_complement}
Let $G$ be a $3$-connected cubic graph and let $\ca O$ be a set of pairwise edge-disjoint cycles. Then, there exists a 1-factor $F$ of $G$ such that $E(F)\cap E(O) \neq \emptyset$ for every $O \in \ca O$.
\end{theo}
 
By Theorem~\ref{theo:odd_circuits_t<r/3}, $\frac{t}{r} \geq \frac{1}{3}$ is a necessary condition to have a positive answer to Question~\ref{ques:t-factor_destroying_arbitrary_circuits}. Analogously to Theorem~\ref{theo:generalisation_KMZ_odd_circuits} and Theorem~\ref{thm:4k-reg_case_Odd_circuits}, in the remainder of this section we prove that $\frac{t}{r} \geq \frac{1}{2}$ with $t$ even and $\frac{t}{r} = \frac{1}{3}$ are also sufficient conditions for $\kappa=3$. 
In order to do so, we need the following slightly stronger version of Theorem~\ref{theo:KMZ_1+factor_acyclic_complement}.

\begin{cor}
\label{cor:KMZ_1+factor_acyclic_complement_2cuts}
Let $G$ be a 2-edge-connected cubic graph. Let $\ca O$ be a set of pairwise edge-disjoint cycles  such that for every $2$-edge-cut $\{e,f\}$, all elements of $\ca O$ are subgraphs of the same component of $G-\{e,f\}$. Then, there exists a 1-factor $F$ of $G$ such that $E(F)\cap E(O) \neq \emptyset$ for every $O \in \ca O$.
\end{cor}

\begin{proof}
Suppose $G$ is a smallest counterexample. By Theorem~\ref{theo:KMZ_1+factor_acyclic_complement}, we may assume $G$ has a 2-edge-cut $\{x_1y_1,x_2y_2\}$. Let $H_1, H_2$ be the components of $G-\{x_1y_1,x_2y_2\}$, where $x_1,x_2$ and all cycles of $\ca O$ belong to $H_1$. By the minimality of $G$, the cubic graph $H_1+x_1x_2$ has a perfect matching $M$ that intersects every cycle of $\ca O$. Furthermore, $H_2 +y_1y_2$ has two perfect matchings $M_1, M_2$ such that $y_1y_2 \in M_1$ and $y_1 y_2 \notin M_2$. If $x_1x_2 \in M$, then $(M \setminus \{x_1x_2\}) \cup (M_1 \setminus \{y_1y_2\}) \cup \{x_1y_1, x_2y_2\}$ is a perfect matching of $G$. If $x_1x_2 \notin M$, then $M \cup M_2$ is a perfect matching of $G$. A contradiction since $G$ is a counterexample.
\end{proof}

By using the above corollary, we obtain similar results as in Section~\ref{sec:13} and Section~\ref{sec:12}. In the remainder of this section, we state these results and shortly explain how the proofs are adjusted. 

Let $t\geq 1$ be an integer, let $G$ be a $3$-connected $3t$-regular graph and let $\mathcal{O}$ be a set of pairwise edge-disjoint cycles of length at least 3 of $G$. Let $G'$ be the cubic graph constructed from $G$ in the proof of Theorem~\ref{theo:generalisation_KMZ_odd_circuits} and let $\mathcal{O}'$ be defined analogously. For every $v,w \in V(G)$, the graph $G'-(V(T_v) \cup V(T_w))$ is connected, since $G$ is $3$-connected. Hence, $G'$ is 2-edge-connected. Moreover, if $\{e,f\}$ is a 2-edge-cut in $G'$, then there are two vertices $v,w \in V(G)$ that are connected by parallel edges in $G$ such that $e \in E(T_v)$ and $f \in E(T_w)$. An example is given in Figure~\ref{fig:2_cut_example}.

\begin{figure}[!htbp]
\centering
\includegraphics[scale=0.5]{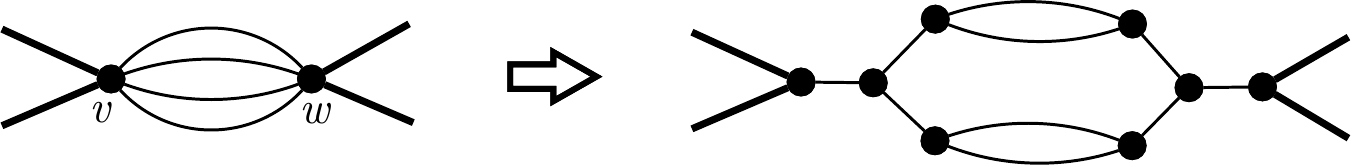}
\caption{An example for the replacement of two vertices $v,w$ of $G$ (left) in the proof of Theorem~\ref{theo:generalisation_KMZ_odd_circuits} in the case $t=2$. The resulting graph $G'$ (right) has 2-edge-cuts. The bold edges belong to a cycle of $\ca O$ or $\ca O'$, respectively.}
\label{fig:2_cut_example}
\end{figure}

Since $\ca O$ does not contain $2$-cycles, all cycles of $\ca O'$ are subgraphs of the same component of $G'-\{e,f\}$. Thus, the following theorem can be proven analogously to Theorem~\ref{theo:generalisation_KMZ_odd_circuits} by using Corollary~\ref{cor:KMZ_1+factor_acyclic_complement_2cuts}.

\begin{theo}
\label{theo:generalisation_KMZ_arbitrary_circuits}
Let $t\geq 1$ be an integer and let $G$ be a $3$-connected $3t$-regular graph. Let $\mathcal{O}$ be a set of pairwise edge-disjoint cycles of length at least 3 of $G$. Then, there exists a $t$-factor $F$ of $G$ such that $E(F) \cap E(O)$ is a non-empty matching of $G$ for every $O \in \mathcal{O}$.
\end{theo}

Note that we can not prescribe an edge that belongs to the $t$-factor as in Theorem~\ref{theo:generalisation_KMZ_odd_circuits}, since this is not possible in Theorem~\ref{theo:KMZ_1+factor_acyclic_complement}.

Next, let $t \geq 2$ be an even integer, let $G$ be a $3$-connected $2t$-regular graph and let $\mathcal{O}$ be a set of pairwise edge-disjoint cycles of length at least 3 of $G$. Let $G'$ be the cubic graph constructed from $G$ in the proof of Lemma~\ref{lem:orientation_4k-reg_case} and let $\mathcal{O}'$ be defined analogously. By the same arguments as above, $G'$ is 2-edge-connected and for every 2-edge-cut $\{e,f\}$ all cycles of $\ca O'$ are subgraphs of the same component of $G'-\{e,f\}$. Thus, by using Corollary~\ref{cor:KMZ_1+factor_acyclic_complement_2cuts}, 
the following lemma can be proven analogously to Lemma~\ref{lem:orientation_4k-reg_case}.

\begin{lem}
\label{lem:orientation_4k-reg_case_arbitrary_cycles}
Let $t \geq 2$ be an even integer, let $G$ be a $3$-connected $2t$-regular graph and let $\mathcal{O}$ be a set of pairwise edge-disjoint cycles of length at least 3 of $G$. Then, there exists an orientation of $G$ such that
\begin{itemize}
\item[$(i)$] every vertex has an even indegree,
\item[$(ii)$] no cycle of $\ca O$ is an oriented cycle.
\end{itemize}
\end{lem}

By using Lemma~\ref{lem:orientation_4k-reg_case_arbitrary_cycles} instead of Lemma~\ref{lem:orientation_4k-reg_case} in the proof of Theorem~\ref{thm:4k-reg_case_Odd_circuits}, we obtain the following theorem.

\begin{theo}
\label{thm:4k-reg_case_arbitrary_circuits}
Let $t \geq 2$ be an even integer and let $G$ be a $3$-connected $2t$-regular graph. Let $\mathcal{O}$ be a set of pairwise edge-disjoint cycles of length at least 3 of $G$. Then, there exists a $t$-factor $F$ of $G$ such that $E(O) \cap E(F) \neq \emptyset$ and $E(O) \cap (E(G) \setminus E(F)) \neq \emptyset$ for every $O \in \mathcal{O}$.
\end{theo}

We remark that, in order to have a positive answer to Question \ref{ques:t-factor_destroying_arbitrary_circuits}, it is necessary that $\frac tr \geq \frac 1\kappa$, for $\kappa=2,3$.

\section{Concluding remarks and open problems}
\label{sec:concluding_remarks}

In this article, we have considered the generalization of a problem that naturally arises from a recent result in the cubic case. We have proved some necessary conditions for a positive answer to Question~\ref{ques:t-factor_destroying_odd_circuits}, which concerns odd-length cycles, and we suspect that these conditions might be sufficient. We leave it as a relevant open problem.

\begin{prob}\label{pro:oddcase}
Prove or disprove the following assertion: let $r,t$ be positive integers such that $\frac{1}{3} \leq \frac{t}{r} \leq 1$. Let $G$ be a $2$-connected $r$-regular graph admitting a $t$-factor and let $\cal O$ be a set of edge-disjoint odd cycles of $G$. Then, there exists a $t$-factor of $G$ intersecting every cycle in $\cal O$ in at least one edge. 
\end{prob}

Please note that the existence of a $2$-factor intersecting all odd cycles of a given set in a $2$-connected $5$-regular graph represents the smallest case not addressed in this paper.

We also remark that in the case $r=3$ and $t=1$ the answer to Question~\ref{ques:t-factor_destroying_odd_circuits} is trivially affirmative for 3-edge-colorable graphs, as every odd cycle intersects all three color classes, and thus the hard part of the proof of Theorem~\ref{theo:KMZ_odd_circuits} was for non-3-edge-colorable graphs. The same distinction does not hold for $r>3$. Indeed, we do not know of any approach that allows us to reduce Problem~\ref{pro:oddcase}  to non-$r$-edge-colorable graphs.

By avoiding $2$-cycles an analogous problem in the case of cycles having arbitrary length also remains. An optimal ratio, if it does exist, seems to depend on the connectivity $\kappa$ of $G$.

\begin{prob}
Prove or disprove the following assertion: let $r,t$ be positive integers and $\kappa=2,3$ such that $\frac{1}{\kappa} \leq \frac{t}{r} \leq 1$. Let $G$ be a $\kappa$-connected $r$-regular graph $G$ admitting a $t$-factor and let $\cal O$ be a set of edge-disjoint cycles of length at least $3$ of $G$. Then, there exists a $t$-factor of $G$ intersecting every cycle in $\cal O$ in at least one edge. 
\end{prob}

Finally, another interesting question is whether the original statement of Conjecture~\ref{con:S4C} (which is true by Theorem~\ref{theo:KMZ_odd_circuits}) can be extended to $r$-graphs of higher regularity. More precisely, the following question seems to be natural. What is the minimum number $t$ such that every $r$-graph has $t$ perfect matchings whose removal leaves a bipartite graph? 

\bibliography{Literatur}{}

@article {BSW_reg_factors_of_reg_graph,
	author = {Bollobás, B. and Saito, Akira and Wormald, N.C.},
	title = {Regular factors of regular graphs},
	year = {1985},
	journal = {J. Graph Theory},
	volume = {9},
	number = {1},
	pages = {97--103},
	doi = {10.1002/jgt.3190090107},
	url = {https://www.scopus.com/inward/record.uri?eid=2-s2.0-84897999588&doi=10.1002%2fjgt.3190090107&partnerID=40&md5=8a5cbff9beb823218a84f61a350b2566},
	type = {Article},
	publication_stage = {Final},
	source = {Scopus},
}

@article{seymour1979multi,
  title={On multi-colourings of cubic graphs, and conjectures of {F}ulkerson and {T}utte},
  author={P. D. Seymour},
  FJournal = {Proceedings of the London Mathematical Society},
  Journal = {Proc. London Math. Soc.},
  volume={3},
  number={3},
  pages={423--460},
  year={1979},
  publisher={Wiley Online Library}
}

@article{fulkerson1971blocking,
	title={Blocking and anti-blocking pairs of polyhedra},
	author={D. R. Fulkerson},
	JOURNAL = {Math. Programming},
    FJOURNAL = {Mathematical Programming},
	volume={1},
	number={1},
	pages={168--194},
	year={1971},
	publisher={Springer}
}

@article{fan1994fulkerson,
	title={{F}ulkerson's Conjecture and Circuit Covers},
	author={G. Fan and A. Raspaud},
	journal={J. Comb. Theory, Ser. B},
	volume={61},
	number={1},
	pages={133--138},
	year={1994},
	publisher={Elsevier}
}

@Article{Meredith,
	Author = {G. H. J. Meredith},
	Title = {Regular {{\(n\)}}-valent {{\(n\)}}-connected non-Hamiltonian non {{\(n\)}}-edge-colourable graphs},
	FJournal = {J. Comb. Theory, Ser. B},
	Journal = {J. Comb. Theory, Ser. B},
	ISSN = {0095-8956},
	Volume = {14},
	Pages = {55--60},
	Year = {1973},
	Language = {English},
	DOI = {10.1016/S0095-8956(73)80006-1},
	zbMATH = {3375518},
	Zbl = {0237.05106}
}

@Article{ReductionBFC,
 Author = {M{\'a}{\v{c}}ajov{\'a}, E. and Mazzuoccolo, G.},
 Title = {Reduction of the {Berge}-{Fulkerson} conjecture to cyclically 5-edge-connected snarks},
 FJournal = {Proceedings of the American Mathematical Society},
 Journal = {Proc. Am. Math. Soc.},
 ISSN = {0002-9939},
 Volume = {148},
 Number = {11},
 Pages = {4643--4652},
 Year = {2020},
 Language = {English},
 DOI = {10.1090/proc/15057}
}

@article{KARDOS20231,
title = {Disjoint odd circuits in a bridgeless cubic graph can be quelled by a single perfect matching},
FJournal = {Journal of Combinatorial Theory. Series B},
Journal = {J. Comb. Theory, Ser. B},
volume = {160},
pages = {1-14},
year = {2023},
issn = {0095-8956},
doi = {https://doi.org/10.1016/j.jctb.2022.12.003},
url = {https://www.sciencedirect.com/science/article/pii/S0095895622001241},
author = {František Kardoš and Edita Máčajová and Jean Paul Zerafa}
}

@article{kardos2025threecuts,
  title={Three-cuts are a charm: acyclicity in 3-connected cubic graphs},
  author={Kardo{\v{s}}, F. and M{\'a}{\v{c}}ajov{\'a}, E. and Zerafa, J.P.},
  journal={Combinatorica},
  volume={45},
  number={1},
  pages={11},
  year={2025},
  publisher={Springer}
}

@article{MAZZUOCCOLO2013235,
title = {New conjectures on perfect matchings in cubic graphs},
journal = {Electron. Notes in Discrete Math.},
volume = {40},
pages = {235-238},
year = {2013},
issn = {1571-0653},
doi = {https://doi.org/10.1016/j.endm.2013.05.042},
url = {https://www.sciencedirect.com/science/article/pii/S1571065313000437},
author = {Giuseppe Mazzuoccolo}
}

@article{MACAJOVA2005112,
title = {Fano colourings of cubic graphs and the {F}ulkerson Conjecture},
journal = {Theor. Comput. Sci.},
volume = {349},
number = {1},
pages = {112-120},
year = {2005},
issn = {0304-3975},
doi = {https://doi.org/10.1016/j.tcs.2005.09.034},
url = {https://www.sciencedirect.com/science/article/pii/S0304397505005724},
author = {M{\'a}{\v{c}}ajov{\'a}, E. and Martin Škoviera}
}

@article{Kaiser2010,
author = {Kaiser, Tomáš and Raspaud, André},
year = {2010},
month = {07},
pages = {1307-1315},
title = {Perfect matchings with restricted intersection in cubic graphs},
volume = {31},
JOURNAL = {European J. Combin.},
FJOURNAL = {European Journal of Combinatorics},
doi = {10.1016/j.ejc.2009.11.007}
}

@article{brinkmann2013generation,
  title={Generation and properties of snarks},
  author={Brinkmann, Gunnar and Goedgebeur, Jan and H{\"a}gglund, Jonas and Markstr{\"o}m, Klas},
  FJournal = {Journal of Combinatorial Theory. Series B},
  Journal = {J. Comb. Theory, Ser. B},
  volume={103},
  number={4},
  pages={468--488},
  year={2013},
  publisher={Elsevier}
}
\addcontentsline{toc}{section}{References}
\bibliographystyle{abbrv}

\end{document}